%% file: Quasi_affine.tex
\documentclass[a4paper,10pt,oneside]{amsart}
\usepackage[a4paper, textwidth=15.5cm]{geometry}
\usepackage{amsmath,amsthm,amssymb,amsfonts,extarrows}
\usepackage{enumerate,amscd,amsxtra,MnSymbol}
\usepackage{mathrsfs}
\usepackage{color}
\usepackage[cmtip,all]{xy}
\usepackage{eucal}
\usepackage{bbold}
\usepackage{ upgreek }
\usepackage{paralist}
\usepackage{tikz-cd}
\usepackage{stackrel}
\usepackage{extarrows}

\usepackage[dvipsnames]{xcolor} 
\usepackage[backref=page]{hyperref}
\hypersetup{
    colorlinks=true,
    linkcolor=PineGreen, 
    citecolor=BurntOrange,
    urlcolor=CadetBlue,
    }

\usepackage{thmtools} % for correct autorefs
\usepackage{quiver}
\usepackage{tikz-cd}
\usepackage{csquotes}
\usepackage{pdfpages}

\input{Comands}

\title{Quasi-affine schemes and singly compactly generated $t$-structures}
\author{Giovanni Rossanigo}

\begin{document}
\maketitle
\begin{abstract}
    We show that  for a quasi-compact quasi-separated scheme $X$ with an ample family of line bundles, the connective half $\QCoh(X)_{\geq0}$ of the standard $t$-structure  on the derived $\infty$-category of quasi-coherent sheaves is compactly generated by a connective perfect object if and only if $X$ is quasi-affine.
\end{abstract}

\begin{figure}[h]
\centering
\includegraphics[scale=0.2]{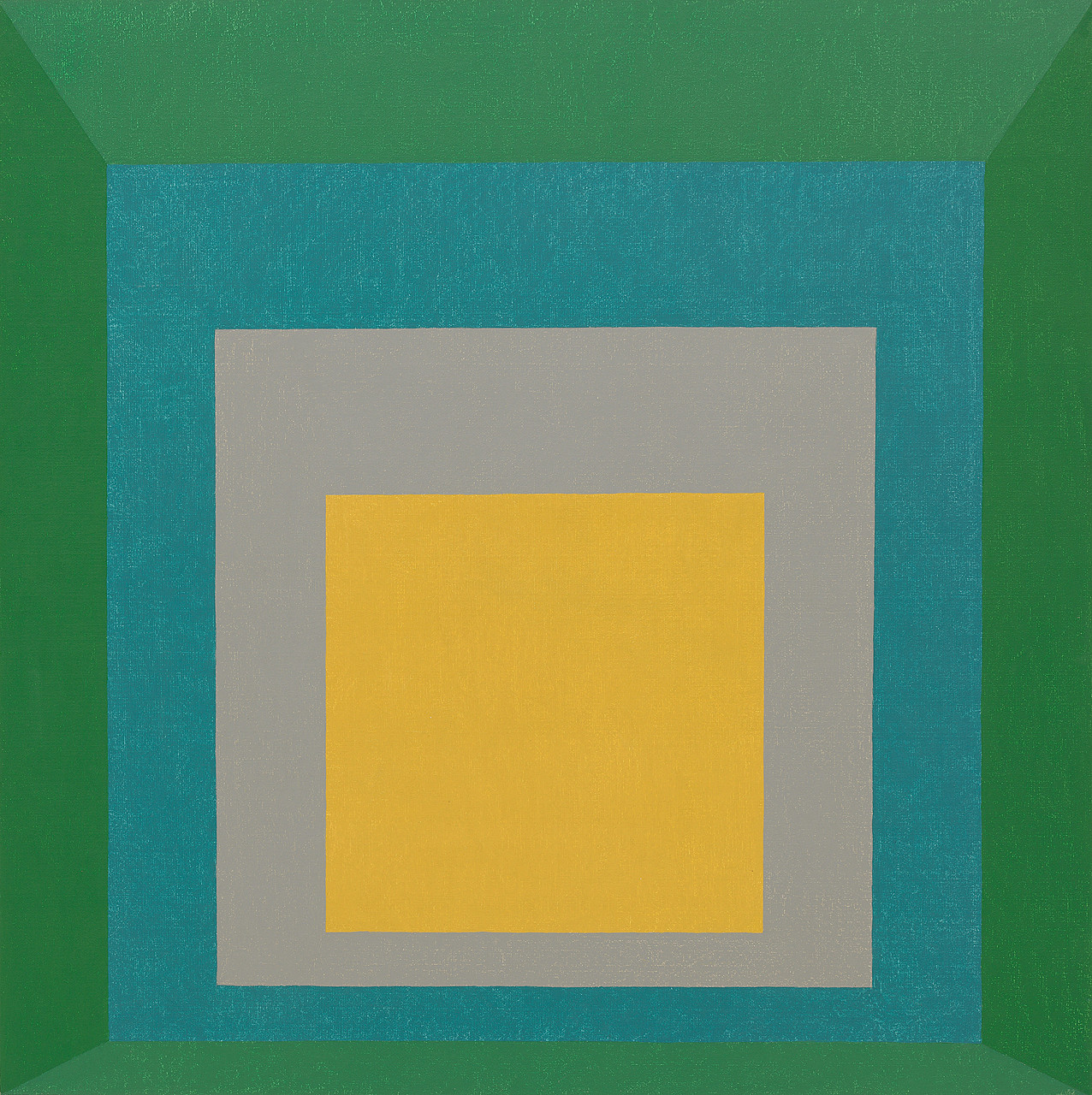}
\caption{\emph{Homage to the square: apparition}, Josef Albers, 1959.}
\end{figure}
\newpage
\tableofcontents

\section{Introduction}
In recent years there has been considerable attention to presentable stable $\infty$-categories equipped with a single compact generator.
For such categories, \cite[Theorem A.1]{alonso2002constructiontstructuresequivalencesderived} says that there exists a $t$-structure whose connective half is compactly generated  by the given single compact generator, in the sense that the connective half is the smallest full subcategory closed under extensions and  small colimits containing the generator. 
A well-known result of Bondal and Van den Bergh \cite[Theorem 3.1.1]{bondal2002generatorsrepresentabilityfunctorscommutative} shows the derived stable $\infty$-category of quasi-coherent sheaves $\QCoh(X)$ on a quasi-compact quasi-separated scheme $X$ admits such compact generator. 
For applications, it suffices to know that this $t$-structure is  \emph{equivalent} (in the sense of \cite[Definition 0.18]{neeman2025triangulatedcategoriessinglecompact}) to the standard $t$-structure, a result proved in \cite[Theorem 3.2]{neeman2024boundedtstructurescategoryperfect}.  
In any case, it is interesting to understand if the  two $t$-structures agree, or, in equivalent terms, if the standard $t$-structure is  compactly generated by a single object.
Our main result shows that, under a mild assumption, this occurs precisely when the scheme is quasi-affine.
\begin{theorem}
    Let $X$ be a quasi-compact quasi-separated scheme.
    Then $X$ is quasi-affine if and only if it admits an ample family of line bundles and the connective half $\QCoh(X)_{\geq 0}$ is compactly generated by a connective perfect object.
\end{theorem}
Notice that every quasi-affine scheme admits an ample line bundle, the structure sheaf.
The key input on the \enquote{geometric} side is the \emph{$1$-resolution property}.
Under the assumption that $X$ admits an ample family of line bundles, \cite[Proposition 3.4]{deshmukh2020quasiaffineness1resolutionproperty} shows that the $1$-resolution property forces $X$ to be quasi-affine.
On the \enquote{categorical} side, starting from a connective perfect generator $G\in \QCoh(X)_{\geq 0}$, the ample-family hypothesis implies that $G$ can be chosen \emph{strict}, that is, represented by a bounded
complex of vector bundles. 
From such a representative we extract an associated vector bundle $E_G=\bigoplus_i E_i$.
We then show that $E_G$ generates the heart $\QCoh(X)^\heartsuit$, and a simple argument implies that every  object of finite-type of $\QCoh(X)^\heartsuit$ is a quotient of a finite sum of copies of $E_G$.
In other words, $E_G$ is a \emph{special} vector bundle and $X$ satisfies the $1$-resolution property, so the conclusion follows from \cite[Proposition 3.4]{deshmukh2020quasiaffineness1resolutionproperty}.

Finally, our result is compatible with recent nonexistence results: for instance, \cite[Theorem~A.1]{bhaduri2025nonexistencesinglycompactlygenerated} proves that if $X$ is a smooth projective curve over the complex numbers, then the standard
connective half $\QCoh(X)_{\geq 0}$ is not singly compactly generated. Since such schemes admit an ample family of line bundles but are not quasi-affine, this phenomenon is predicted by our characterization.

\subsection*{Organization.}
In \autoref{section: prestable} we recall basic facts on Grothendieck prestable $\infty$-categories and recall the notion of finite type objects in the heart. 
In \autoref{section: 1 resolution} we relate strict compact generators to the $1$-resolution
property and prove the main theorem. 
In \autoref{section: examples} we  discuss the optimality of the main result.

\subsection*{Conventions}
We freely use the language of stable $\infty$-categories and $t$-structures on them, which we  grade homologically. 
We denote by $\hom_\mathcal{C}(-,-)$ the mapping spectrum of a  stable $\infty$-category $\mathcal{C}$.
Given a scheme $X$, we denote by $\QCoh(X)$ the derived stable $\infty$-category of quasi-coherent sheaves on $X$.
We will also denote by $\QCoh(X)_{\geq0}$ the connective half of the standard $t$-structure, and by $\QCoh(X)^\heartsuit$ the Grothendieck abelian $1$-category of discrete quasi-coherent sheaves on $X$. 
If $A$ is a ring, we will denote by $\Mod_A$ the  derived stable $\infty$-category of modules over $A$.

\subsection*{Acknowledgments} 
I wrote this note back in the late winter, during repeated trips between Milan and Turin, all posthumous to February 5th, the day my uncle \emph{Andrea} lost his battle against cancer. 
I dedicate these pages to his memory, and to everyone who continues to face this illness with courage.

\section{Grothendieck prestable $\infty$-categories}\label{section: prestable}
We begin with some reminders on Grothendieck prestable categories that the reader can find in \cite[Appendix C]{Lurie-SAG}.
\begin{remark}
    Recall that  an $\infty$-category  is \emph{prestable} if it is pointed, it admits finite colimits, the suspension functor is fully-faithful and every span $0\to \susp x\xleftarrow{} y$ can be completed to a pullback, and the resulting square is also a pushout. 
    Among prestable categories, the one admitting finite limits may be characterized as the connective half of a non-canonical $t$-structure on certain stable $\infty$-categories (usually taken to be the Spanier-Whitehead stabilization or the spectrum objects stabilization).
    Given a prestable $\infty$-category $\mathcal{C}$, there is a notion of \emph{heart} $\mathcal{C}^\heartsuit$, defined to be the full subcategory of discrete objects.
    In the case where $\mathcal{C}$ admits finite limits, the heart $\mathcal{C}^\heartsuit$ may be identified with the heart of the non-canonical $t$-structure. 
\end{remark}
\begin{remark}
    Recall that a prestable $\infty$-category $\mathcal{C}$ is \emph{Grothendieck prestable} if it is presentable and filtered colimits are left exact.
    In this case, the heart $\mathcal{C}^\heartsuit$ is a Grothendieck abelian $1$-category. 
    Furthermore, every Grothendieck abelian $1$-category arises as the heart of a Grothendieck prestable $\infty$-category. 
    A great number of examples of Grothendieck prestable $\infty$-categories is given by prestable compactly generated $\infty$-categories. 
\end{remark}
The main example for us is the following.
\begin{example}
    Let $X$ be a quasi-compact quasi-separated scheme and let $\QCoh(X)_{\geq0}$ be the connective half of the standard $t$-structure on  the derived $\infty$-category $\QCoh(X)$ of quasi-coherent sheaves.
    Then $\QCoh(X)_{\geq0}$ is Grothendieck prestable, being compactly-generated by the connective perfect objects. 
    See \cite[Proposition 9.6.1.2]{Lurie-SAG}.
\end{example}
The problem of compact generation of $t$-structure is well understood.
\begin{remark}\label{remark: comp gene}
     Let $\mathcal{C}$ be a presentable stable $\infty$-category and let $E$ be an object. 
     Then \cite[Proposition 1.4.4.11]{Lurie-HA} states that there exists an accessible $t$-structure whose connective half  $\langle E\rangle_{\geq0}$ is the smallest full subcategory which contains $E$ and is closed under extensions and small colimits. 
     If $E$ is furthermore compact, then $\langle E\rangle_{\geq0}$ is Grothendieck prestable, since it is compactly generated by the closure under retracts of the smallest full subcategory which contains $E$ and is closed under extensions and finite colimits.
     In other terms, $E$ is a compact generator for $\langle E\rangle_{\geq0}$ and the $t$-structure is \emph{singly compactly generated}.
\end{remark}
In the case of compact generators, the previous observation can be improved.
\begin{remark}\label{remark: comp gene 1}
    Let $\mathcal{C}$ be a presentable stable $\infty$-category with a single compact generator $G$. 
    Then \cite[Proposition C.6.3.1]{Lurie-SAG} states that there exists an accessible, compatible with filtered colimits, right-complete $t$-structure whose connective half $\langle G\rangle_{\geq0}$ is the smallest full subcategory which contains $G$ and is closed under extensions and small colimits. 
    This connective half is Grothendieck prestable (being compactly generated by the closure under retracts of the smallest full subcategory which contains $G$ and is closed under extensions and finite colimits). 
    Thus $G$ is a compact generator for $\langle G\rangle_{\geq0}$.
    The negative half is instead given by those objects $x\in\mathcal{C}$ such that $\pi_{n}\hom_\mathcal{C}(G,x)\simeq0$ for $n\geq 1$.
\end{remark}

\begin{remark}\label{remark: generator and heart}
    Let $\mathcal{C}$ be a Grothendieck prestable $\infty$-category. 
    Recall that a generating subcategory is a full subcategory $\mathcal{C}_0\subseteq\mathcal{C}$ such that for every object $x\in\mathcal{C}$,  there exists a collection of maps $p_i: C_i\to x$, where each $C_i$ belongs to $\mathcal{C}_0$ and the induced  map $\oplus_{i\in I}C_i\to x$ induces a $\pi_0$-epimorphism in $\mathcal{C}^\heartsuit$.
    Notice that if $\mathcal{C}_0\subseteq\mathcal{C}$ generates under colimits and extensions, then it is generating.
    Particular importance is the case where $\mathcal{C}_0$ consists of a single object $G$ (which is not necessarily compact); this is the case of \autoref{remark: comp gene} and \autoref{remark: comp gene 1}.
    In this case the truncation $\pi_0G$ is a generator for the Grothendieck abelian $1$-category $\mathcal{C}^\heartsuit$.
\end{remark}

We conclude this section with an observation.
\begin{definition}
     Let $\mathcal{C}^\heartsuit$ be a Grothendieck abelian $1$-category.
     An object $x\in\mathcal{C}^\heartsuit$ is called \emph{of finite type} if it satisfies the following equivalent conditions:
    \begin{enumerate}
    \item For every filtered diagram $(y_i)_{i\in I}$ in $\mathcal{C}^\heartsuit$ such that each transition map
    $y_i\to y_j$ is a monomorphism, the canonical map $\colim_{i\in I}\Hom_{\mathcal{C}^\heartsuit}(x,y_i)\to\Hom_{\mathcal{C}^\heartsuit}(x,\colim_{i\in I} y_i)$ is an isomorphism.
    \item If $(x_i)_{i\in I}$ is a filtered system of subobjects $x_i\subseteq x$ whose union is $x$, then $x=x_{i_0}$ for some $i_0\in I$.
    \end{enumerate}
    The equivalence of these conditions is shown in \cite[Theorem 1.8]{article}.
\end{definition}

We recall the following well known result.
\begin{lemma}\label{lemma: generators in Grothendieck abelian and finite type}
    Let $\mathcal{C}^\heartsuit$ be a Grothendieck abelian $1$-category and let $E\in\mathcal{C}^\heartsuit$ be a generator. 
    Then every object $x\in\mathcal{C}^\heartsuit$ of finite type admits an epimorphism $\oplus_{j\in J_0}E\to x$ from a finite direct sum.
\end{lemma}
\begin{proof}
    Since $E$ is a generator, there exists a set $I$ and an epimorphism $p:\bigoplus_{i\in I} E \twoheadrightarrow x$. 
    Define  $x_J:= \mathrm{im}(\bigoplus_{j\in J} E \xrightarrow{p} x)\subseteq x$ for every finite subset $J\subseteq I$.
    If $J\subseteq J'$, then $x_J\subseteq x_{J'}$, so the poset of finite subsets of $I$ indexes a filtered system of subobjects $(x_J)$ of $x$.
    Since $p$ is an epimorphism, the object $x$ is generated by the images of the individual summands $E\to \bigoplus_{i\in I}E \xrightarrow{p} x$, and any finite sum of such images is precisely some $x_J$. Thus  $\colim_J x_J \simeq x$, that is, the union of the $x_J$ is all of $x$.
    Now, since $x$ is of finite type, it follows that $x\simeq x_{J_0}$ for some finite $J_0\subset I$. By construction, the map $\bigoplus_{j\in J_0}E \to x$ has image $x_{J_0}\simeq x$, hence it is an epimorphism.
\end{proof}

\section{The $1$-resolution property}\label{section: 1 resolution}
Recall that a scheme $X$ has the \emph{resolution property} if every discrete quasi-coherent module of finite type is the quotient of a vector bundle. 
\begin{example}\label{example: ample family implies resolution property}
    Let $X$ be a scheme. 
    If $X$ has an ample family of line bundles, then  \cite[\href{https://stacks.math.columbia.edu/tag/0GMM}{Lemma 0GMM}]{stacks-project} implies that $X$ has the resolution property. 
\end{example}
\begin{remark}\label{remark: strict}
    For us   \cite[\href{https://stacks.math.columbia.edu/tag/0F8F}{Lemma 0F8F}]{stacks-project} is the main result: it states that, for a quasi-compact quasi-separated %\footnote{Here the separated assumption is needed to force $\QCoh$ to be equivalent to $D_\text{qc}$} 
    scheme with the resolution property, every perfect complex can be represented by a bounded complex of vector bundles. 
    In other terms, every perfect complex is \emph{strict}.
\end{remark}

We now need a stronger resolution property, which resolves every quasi-coherent module of finite type by a single vector bundle.
\begin{definition}
    Let $X$ be a quasi-compact quasi-separated scheme.
    We will say that $X$ has the \emph{$1$-resolution property} if there exists a vector bundle $E$ such that every discrete quasi-coherent sheaf of finite type on $X$ is a quotient of $\oplus_{i\in I}E$ for some finite set $I$.
    In this case, $E$ is called a \emph{special vector bundle}.
\end{definition}

\begin{remark}\label{remark: proposition 3.4 paper}
    Let $X$ be a quasi-compact scheme admitting an ample family of line bundles. 
    If $X$ has the $1$-resolution property then it is  quasi-affine.
    This is \cite[Proposition 3.4]{deshmukh2020quasiaffineness1resolutionproperty}.
\end{remark}

\begin{construction}
    Let $X$ be a  quasi-compact quasi-separated scheme and let $G\in\QCoh(X)_{\geq0}$ be a strict compact generator. 
    Pick an equivalence 
    \[
    G\simeq (0\to E_d\to \dots\to E_1\to E_0\to 0)
    \]
    with a bounded complex of vector bundles. 
    Let $E_G=\bigoplus_{i=0}^dE_i\in\QCoh(X)^\heartsuit$ and notice that $E_G$ depends on the choice of the above equivalence.
    We will refer to $E_G$ as a \emph{vector bundle associated to $G$}.
\end{construction}

\begin{lemma}\label{lemma: E_G generator}
     Let $X$ be a  quasi-compact quasi-separated scheme. 
     Let $G\in\QCoh(X)_{\geq0}$ be a strict compact generator for the connective half and let $E_G$ be an associated vector bundle. Then:
     \begin{enumerate}
        \item There is an equality $\QCoh(X)_{\geq0}=\langle E_G\rangle_{\geq0}$.
        \item The object $E_G$ is a generator of $\QCoh(X)^\heartsuit$. 
     \end{enumerate}
\end{lemma}
\begin{proof}
    Consider point $(1)$. 
    Since $\QCoh(X)_{\geq0}$ is closed under extension and small colimits and contains $E_G$, it follows that $\langle E_G\rangle_{\geq0}\subseteq \QCoh(X)_{\geq0}=\langle G\rangle_{\geq0}$.
    For the converse, it suffices to show that $G$ lies in $\langle E_G\rangle_{\geq0}$.
    The proof now goes by induction on $0\leq i\leq d$, where $d$ is the coconnectivity of $G$.
    Let $\sigma_{\leq i}$ be the stupid truncation  and consider the cofibre sequence
    \[
     \sigma_{\leq i-1}G\to\sigma_{\leq i}G\to \susp^i E_i.
    \]
    If $i=0$ then the connectivity of $G$ implies that $\sigma_{\leq -1}G\simeq 0$, so that $\sigma_{\leq G}\simeq E_0$ belongs to $\langle E_G\rangle_{\geq0}$, being $E_0$ a retract of $E_G$.
    The inductive hypothesis assumes now $\sigma_{\leq i-1}G\in \langle E_G\rangle_{\geq0}$. 
    Since the last term belongs to $\langle E_G\rangle_{\geq0}$, being a suspension of a retract of $E_G$, it follows that $\sigma_{\leq i}G\in\langle E_G\rangle_{\geq0}$, being $\langle E_G\rangle_{\geq0}$ closed under extensions. 
    Thus $G\simeq\sigma_{\leq d}G$ belongs to $\langle E_G\rangle_{\geq0}$.

    For $(2)$, since point $(1)$ shows that $E_G$ is a generator for $\QCoh(X)_{\geq0}$, it suffices to apply \autoref{remark: generator and heart} to deduce that $E_G\simeq \pi_0E_G$ is a generator for $\QCoh(X)^\heartsuit$.
\end{proof}
The following is the main result of these pages. 
\begin{lemma}\label{lemma: 1-resolution property}
    Let $X$ be a  quasi-compact quasi-separated scheme. 
    Let $G\in\QCoh(X)_{\geq0}$ be a strict compact generator and let $E_G$ be the associated vector bundle.
    Then $E_G$ is a special vector bundle for $\QCoh(X)^\heartsuit$.
    In particular, $X$ has the $1$-resolution property.
\end{lemma}
\begin{proof}
    Since \autoref{lemma: E_G generator} implies that the object $E_G\in\QCoh(X)^\heartsuit$ is a generator, an application of \autoref{lemma: generators in Grothendieck abelian and finite type} shows that every discrete quasi-coherent sheaf of finite type admits a surjection from a finite sum of copies of $E_G$.
    Since $E_G$ is a vector bundle by definition, the claim follows.
\end{proof}
\begin{corollary}
    Let $X$ be a  quasi-compact quasi-separated scheme and assume that the standard $t$-structure is generated by  a strict connective perfect complex.
    %Let $G\in\QCoh(X)_{\geq0}$ be a strict compact generator and let $E_G$ be the associated vector bundle.
    \begin{enumerate}
        \item If $X$ is a noetherian normal scheme which is J-2, then it is quasi-affine.
        \item If $X$ is a normal noetherian $\mb{Q}$-factorial scheme (for example, regular), then it is quasi-affine.
    \end{enumerate}
\end{corollary}
\begin{proof}
    Standing \autoref{lemma: 1-resolution property}, point $(1)$ follows from \cite[Theorem 3.13]{deshmukh2020quasiaffineness1resolutionproperty} and point $(2)$ from \cite[Corollary 3.8]{deshmukh2020quasiaffineness1resolutionproperty}.
\end{proof}

We summarize our discussion with the following.
\begin{theorem}\label{theorem: main}
    Let $X$ be a quasi-compact quasi-separated scheme.
    Then $X$ is quasi-affine if and only if it admits an ample family of line bundles and the connective half $\QCoh(X)_{\geq 0}$ is compactly generated by a connective perfect object.
\end{theorem}
\begin{proof}
    Let $X$ be quasi-affine and let $\Gamma(X,\mathcal{O}_X)$ denote the ordinary ring of global sections of the structure sheaf.
    Then the canonical map $j:X\into\text{Spec}(\Gamma(X,\mathcal{O}_X))$ is an open immersion.
    It follows that $j^*:\Mod_{\Gamma(X,\mathcal{O}_X)}\to\QCoh(X)$ is a $t$-exact localization, and hence $\mathcal{O}_X\simeq j^*(\mb{1}_{\Gamma(X,\mathcal{O}_X)})$ compactly generates $\QCoh(X)_{\geq0}$. 
    Since the monoidal unit is ample, one direction follows.
    
    Conversely, assume that $\QCoh(X)_{\geq0}$ is compactly generated by a connective perfect object. 
    Since $X$ admits an ample family of line bundles, this  connective perfect object is strict by \autoref{example: ample family implies resolution property} and \autoref{remark: strict}.
    By \autoref{lemma: 1-resolution property} the scheme $X$ satisfies the $1$-resolution property, and since it has an ample family of line bundles, \autoref{remark: proposition 3.4 paper} implies that $X$ is quasi-affine.
\end{proof}

\section{Final observations}\label{section: examples}
The main result recovers an \enquote{horrible discovery}.
\begin{example}
    In \cite[Theorem A.1]{bhaduri2025nonexistencesinglycompactlygenerated} it is stated that  a smooth projective curve over the complex numbers has connective half which is not singly compactly generated. 
    This is an example of a scheme with an ample family of line bundles (and actually, a single one), for which \autoref{theorem: main}  applies. 
    Notice furthermore that \autoref{theorem: main} provides a different proof (since \cite[Theorem A.1]{bhaduri2025nonexistencesinglycompactlygenerated} relies on the Harder-Narasimhan filtration).
\end{example}
%Notice that \autoref{theorem: main} is also far from being optimal: \cite[Proposition 3.7]{bhaduri2025nonexistencesinglycompactlygenerated} shows that algebraic spaces proper over a field with singly compactly generated connective half are precisely the artinian affine schemes.
%It raises the question: are there quasi-compact quasi-separated schemes with singly compactly generated connective half which are not quasi-affine?
One question remains: are there quasi-compact quasi-separated schemes with singly compactly generated connective half which are not quasi-affine?

% Bibliography
\bibliographystyle{alpha}
\bibliography{Bibliography}
\end{document}

%% file: Comands.tex
\theoremstyle{definition}
\newtheorem{theorem}{Theorem}[section]
\newtheorem{definition}[theorem]{Definition}

\newtheorem{corollary}[theorem]{Corollary}
\newtheorem{lemma}[theorem]{Lemma}
\newtheorem{remark}[theorem]{Remark}

\newtheorem{example}[theorem]{Example}

\newtheorem{construction}[theorem]{Construction}

\newtheorem*{theorem*}{Theorem}
\newtheorem*{corollary*}{Corollary}
\newtheorem*{definition*}{Definition}
\newtheorem*{remark*}{Remark}
\newtheorem*{question*}{Question}

\newcommand{\rom}[1]{\uppercase\expandafter{\romannumeral #1\relax}}

\makeatletter
\@ifpackageloaded{hyperref}%
  {\newcommand{\mylabel}[2]% #1=name, #2 = contents
    {\protected@write\@auxout{}{\string\newlabel{#1}{{#2}{\thepage}%
      {\@currentlabelname}{\@currentHref}{}}}}}%
  {\newcommand{\mylabel}[2]% #1=name, #2 = contents
    {\protected@write\@auxout{}{\string\newlabel{#1}{{#2}{\thepage}}}}}
\makeatother

\usepackage{enumitem}
\newenvironment{enumerate*}
  {\begin{enumerate}[label=(\arabic*),topsep=-\parskip+\jot,itemsep=-\parskip+\jot]}
  {\end{enumerate}}

\usepackage{bbm}
\let\mb\mathbbm

\usepackage[utf8]{inputenc}
\DeclareFontFamily{U}{min}{}
\DeclareFontShape{U}{min}{m}{n}{<-> udmj30}{}

\DeclareMathOperator\Hom{Hom}

\def\into{\hookrightarrow}

\def\susp{\Sigma}

\DeclareMathOperator\colim{colim}

\def\phi{\varphi}
\def\epsilon{\varepsilon}

\def\QCoh{\text{QCoh}}

\def\Mod{\text{Mod}}

%% file: Bibliography.bib
@book{Lurie-HA,
	title        = {Higher Algebra},
	author       = {Jacob Lurie},
	year         = 2017,
	month        = {September},
	publisher    = {Unpublished},
	note         = {Available online at \href{https://www.math.ias.edu/~lurie/papers/HA.pdf}{Higher Algebra}}
}

@book{Lurie-SAG,
	title        = {Spectral Algebraic Geometry},
	author       = {Jacob Lurie},
	year         = 2018,
	journal      = {preprint},
	publisher    = {Unpublished},
	note         = {Available online at \href{https://www.math.ias.edu/~lurie/papers/SAG-rootfile.pdf}{Spectral Algebraic Geometry}}
}

@article{stacks-project,
	title        = {\textit{Stacks Project}},
	author       = {The Stacks Project Authors},
	year         = {2018},   
	note      = {Available online at \href{https://stacks.math.columbia.edu}{The Stacks Project}},
}

@misc{alonso2002constructiontstructuresequivalencesderived,
      title={Construction of t-structures and equivalences of derived categories}, 
      author={Leovigildo Alonso and Ana Jeremias and Ma. -Jose Souto},
      year={2002},
      eprint={math/0203009},
      archivePrefix={arXiv},
      primaryClass={math.RT},
      url={https://arxiv.org/abs/math/0203009}, 
}

@misc{bondal2002generatorsrepresentabilityfunctorscommutative,
      title={Generators and representability of functors in commutative and noncommutative geometry}, 
      author={Alexei Bondal and Michel Van den Bergh},
      year={2002},
      eprint={math/0204218},
      archivePrefix={arXiv},
      primaryClass={math.AG},
      url={https://arxiv.org/abs/math/0204218}, 
}

@misc{neeman2024boundedtstructurescategoryperfect,
      title={Bounded t-structures on the category of perfect complexes}, 
      author={Amnon Neeman},
      year={2024},
      eprint={2202.08861},
      archivePrefix={arXiv},
      primaryClass={math.AG},
      url={https://arxiv.org/abs/2202.08861}, 
}

@misc{neeman2025triangulatedcategoriessinglecompact,
      title={Triangulated categories with a single compact generator and two Brown representability theorems}, 
      author={Amnon Neeman},
      year={2025},
      eprint={1804.02240},
      archivePrefix={arXiv},
      primaryClass={math.CT},
      url={https://arxiv.org/abs/1804.02240}, 
}

@misc{deshmukh2020quasiaffineness1resolutionproperty,
      title={Quasi-affineness and the 1-Resolution Property}, 
      author={Neeraj Deshmukh and Amit Hogadi and Siddharth Mathur},
      year={2020},
      eprint={1809.05270},
      archivePrefix={arXiv},
      primaryClass={math.AG},
      url={https://arxiv.org/abs/1809.05270}, 
}

@misc{bhaduri2025nonexistencesinglycompactlygenerated,
      title={Nonexistence of singly compactly generated $t$-structures for schemes}, 
      author={Anirban Bhaduri and Timothy De Deyn and Michal Hrbek and Pat Lank and Kabeer Manali-Rahul},
      year={2025},
      eprint={2511.01622},
      archivePrefix={arXiv},
      primaryClass={math.AG},
      url={https://arxiv.org/abs/2511.01622}, 
}

@article{article,
author = {Garkusha, Grigory},
year = {1999},
month = {10},
pages = {2002},
title = {Grothendieck Categories},
volume = {13},
journal = {St Petersburg Mathematical Journal}
}
